\documentclass{article}

\usepackage{amssymb,amsmath,color,enumerate,ascmac,latexsym,diagbox}
\usepackage[all]{xy}
\usepackage[abbrev]{amsrefs}
\usepackage[mathscr]{eucal}
\usepackage[top=30truemm,bottom=30truemm,left=25truemm,right=25truemm]{geometry}
\usepackage{mathrsfs}
\usepackage{amsthm}
\usepackage[OT2,T1]{fontenc}
\usepackage[russian,english]{babel}

\newcommand{\Q}{\mathbb{Q}}

\newcommand{\Z}{\mathbb{Z}}

\newcommand{\Co}{\mathbb{C}}

\newcommand{\five}{~~~~~}
\newcommand{\ten}{~~~~~~~~~~}

\newcommand{\twen}{~~~~~~~~~~~~~~~~~~~~}
\newcommand{\mm}[1]{\mathop{\mathrm{#1}}}

\newcommand{\ta}[1]{\tanh^{#1}{(t/2)}}

\newcommand{\dv}[2]{\frac{d^{#1}}{{d#2}^{#1}}}

\newcommand{\apoly}[2]{\mathscr{A}\left(#1;#2\right)}
\newcommand{\ctext}[1]{\raise0.2ex\hbox{\textcircled{\scriptsize{#1}}}}
\newcommand{\cy}{\textup{\foreignlanguage{russian}{\cyrsh}}}

\usepackage{graphicx}
\usepackage{stmaryrd}
\makeatletter
\def\mapstofill@{%
	\arrowfill@{\mapstochar\relbar}\relbar\rightarrow}
\newcommand*\xmapsto[2][]{%
	\ext@arrow 0395\mapstofill@{#1}{#2}}

\usepackage[dvipdfmx]{hyperref}

\usepackage{xcolor}
\hypersetup{
	colorlinks=true,
	citebordercolor=green,
	linkbordercolor=red,
	urlbordercolor=cyan,
}

\newtheorem{theorem}{Theorem}[section]

\newtheorem{lemma}[theorem]{Lemma}
\newtheorem{proposition}[theorem]{Proposition}

\newtheorem{remark}{Remark}[section]
\newtheorem{example}{Example}[section]

\makeatletter
\def\underbrace@#1#2{\vtop {\m@th \ialign {##\crcr $\hfil #1{#2}\hfil $\crcr \noalign {\kern 3\p@ \nointerlineskip }\upbracefill \crcr \noalign {\kern 3\p@ }}}}
\def\overbrace@#1#2{\vbox {\m@th \ialign {##\crcr \noalign {\kern 3\p@ }\downbracefill \crcr \noalign {\kern 3\p@ \nointerlineskip }$\hfil #1 {#2}\hfil $\crcr }}}

\def\underbrace#1{%
	\mathop{\mathchoice{\underbrace@{\displaystyle}{#1}}
		{\underbrace@{\textstyle}{#1}}
		{\underbrace@{\scriptstyle}{#1}}
		{\underbrace@{\scriptscriptstyle}{#1}}}\limits
}
\def\overbrace#1{%
	\mathop{\mathchoice{\overbrace@{\displaystyle}{#1}}
		{\overbrace@{\textstyle}{#1}}
		{\overbrace@{\scriptstyle}{#1}}
		{\overbrace@{\scriptscriptstyle}{#1}}}\limits
}

\makeatother

\allowdisplaybreaks

\begin{document}
	
\title{On explicit relations for values of Kaneko-Tsumura's $\lambda$ function}
\author{Kyosuke Nishibiro}
\date{}
\maketitle

\begin{abstract}
	Recently, Kaneko and Tsumura introduced multiple $\widetilde{T}$-values, another kind of poly-Euler numbers and the related Arakawa-Kaneko type zeta function. It is shown that each of them satisfies similar formulas to those of multiple zeta values, poly-Bernoulli numbers and the related Arakawa-Kaneko type zeta function. In this paper, we show some explicit relations for values of Kaneko-Tsumura $\lambda$ function at positive integers, including the duality type relation.
\end{abstract}

\section{Introduction}

First of all, we recall some notations. For an index $\Bbbk=(k_1, \ldots,k_r)\in\Z_{\geq1}^r$, its weight and depth are defined by
\begin{align*}
	\mm{wt}(\Bbbk)=k_1+\cdots+k_r,~~~\mm{dep}(\Bbbk)=r.
\end{align*}
An admissible index is the index with $k_r\geq2$ and a non-admissible index is that with $k_r=1$. For an admissible index $\Bbbk$, we set $\Bbbk_{-}=(k_1,\ldots,k_r-1)$. Also, for $j\in\Z_{\geq1}$ and $l\in\Z_{\geq0}$, we set $\{j\}_{l}=(\underbrace{j, j,\ldots, j}_{l})$. If $l=0$, we regard $\{j\}_0$ as the empty index $\phi$.

In \cite{KT2}, Kaneko and Tsumura introduced multiple $T$-values
\[
T(k_1, \ldots, k_r)=2^r\sum_{\substack{0<m_1<\cdots<m_r\\m_j\equiv j\bmod{2}}}^{}\frac{1}{m_1^{k_1}\cdots m_r^{k_r}}~~~~~~( (k_1,\ldots, k_r)\in\Z_{\geq1}^r, k_r\geq2).
\]
Multiple $T$-values are regarded as level two analogues of multiple zeta values
\[
\zeta(k_1\ldots,k_r)=\sum_{0<m_1<\cdots<m_r}^{}\frac{1}{m_1^{k_1}\cdots m_r^{k_r}}.
\]
Also, a notation of multiple $T$-values is regarded as a `shuffle counterpart' of Hoffman's multiple $t$-values
\[
t(k_1, \ldots, k_r)=2^r\sum_{\substack{0<m_1<\cdots<m_r\\m_j\equiv 1\bmod{2}}}^{}\frac{1}{m_1^{k_1}\cdots m_r^{k_r}}.
\]
For each value, it is known that each vector space spanned by each value over $\Q$ has interesting properties (for details, see \cite{AK, KT2, Hof1}). On the other hand, there are much unknown on them, including dimension conjecture for each vector space. Hence finding relations for each values is important. One method is to consider Arakawa-Kaneko type zeta functions defined by
\begin{align}\label{defxi}
	\xi(k_1,\ldots, k_r;s)=\frac{1}{\Gamma(s)}\int_{0}^{\infty}t^{s-1}\frac{\mm{Li}(k_1,\ldots, k_r;1-e^{-t})}{e^t-1} dt~(\mm{Re}(s)>0)
\end{align}  
and 
\begin{align}\label{defpsi}
	\psi(k_1,\ldots, k_r;s)=\frac{1}{\Gamma(s)}\int_{0}^{\infty}t^{s-1}\frac{\mm{A}(k_1,\ldots,k_r; \ta{})}{\sinh{t}} dt~(\mm{Re}(s)>0)
\end{align}
respectively. Here, 
\begin{align*}
	\mm{Li}(k_1,\ldots, k_r;z)=\sum_{0<m_1<\cdots<m_r}^{}\frac{z^{m_r}}{m_1^{k_1}\cdots m_r^{k_r}}~(|z|<1)
\end{align*}
is the multiple polylogarithm function and
\begin{align*}
	\mm{A}(k_1,\ldots,k_r;z)=2^r\sum_{\substack{0<m_1<\cdots<m_r\\m_j\equiv j\bmod{2}}}^{}\frac{z^{m_r}}{m_1^{k_1}\cdots m_r^{k_r}}~(|z|<1)
\end{align*}	
is the multiple polylogarithm function of level two. Note that
\begin{align*}
	\mm{Li}(1;z)=-\log(1-z),~~~ \mm{A}(1;z)=2\tanh^{-1}{z},
\end{align*}	
and
\begin{align*}
	\mm{Li}(1;1-e^{-t})=t,~~~ \mm{A}(1;\ta{})=t
\end{align*}
hold. In particular, we have
\[
\xi(1;s)=s\zeta(s+1),~~\psi(1;s)=(1-2^{-s})\zeta(s).
\]
It is known that the $\xi$ and $\psi$ functions can be continued analytically to $\Co$ and the values of $\xi$ and $\psi$ functions at positive integers are related to multiple zeta values and multiple $T$-values, respectively (for details, see \cite{AK, KT2, KT3}).

Recently, Kaneko and Tsumura introduced multiple $\widetilde{T}$ values and related zeta function
\begin{align*}
	\widetilde{T}(k_1, \ldots, k_r)=2^r\sum_{\substack{0<m_1<\cdots<m_r\\m_j\equiv j\bmod{2}}}^{}\frac{(-1)^{(m_r-r)/2}}{m_1^{k_1}\cdots m_r^{k_r}}
\end{align*}
and
\begin{align}\label{lambda}
	\lambda(k_1,\ldots, k_r;s)=\frac{1}{\Gamma(s)}\int_{0}^{\infty}t^{s-1}\frac{\apoly{k_1,\ldots,k_r}{\tanh{(t/2+\pi i/4)}}}{\cosh{t}} dt
	\end{align}
	in \cite{KT1}. Here, for $\Bbbk_r=(k_1,\ldots,k_r)\in\Z_{\geq0}^r$, $\apoly{k_1,\ldots, k_r}{z}$ is the multiple polylogarithm function of level four defined by 
	\[
	\apoly{k_1,\ldots,k_{r-1},k_r}{z}= \begin{cases}
		\int_{i}^{z}\frac{1}{u}\apoly{k_1,\ldots,k_{r-1},k_r-1}{u}du &(\Bbbk_r:\mbox{admissible}) ,\\
		\int_{i}^{z}\frac{2}{1-u^2}\apoly{k_1,\ldots,k_{r-1}}{u}du &(\Bbbk_r:\mbox{non-admissible})\end{cases}
	\]
    (from this, we can show that the shuffle product holds for $\apoly{\Bbbk_r}{z}$). In particular, we have
    \begin{align}\label{apolya}
    \apoly{1}{z}=\mm{A}(1;z)-\frac{\pi i}{2}=\tanh^{-1}{\left(\frac{z}{2}+\frac{\pi i}{4}\right)}.
    \end{align}
    Note that multiple $\widetilde{T}$-values are related to certain multiple $L$-values. Let $\chi_4$ be the Dirichlet character of conductor four. Then multiple $\widetilde{T}$-values can be written by multiple $L$-values as
    \[
    \widetilde{T}(k_1, \ldots, k_r)=2^rL_{\cy}(k_1,\ldots,k_r;\chi_4,\ldots,\chi_4).
    \]
    (for details of multiple $L$ values, see \cite{AK2}). Hence multiple $\widetilde{T}$-values are also called multiple $L$-values of level four. In \cite{KT1}, some relations between $\lambda$ function and multiple $\widetilde{T}$-values are showed. In this paper, we show other relations on them. In section $2$, we show some explicit formulas for values of $\lambda$ function. In section $3$, we show duality type relation for $\lambda$ function.

\section{Explicit formulas for values of Kaneko-Tsumura's $\lambda$ function}

In this section, we show some explicit formulas for values of $\lambda$ function and multiple $\widetilde{T}$-values. First of all, we show that the values of $\lambda$ function can be written in terms of multiple $\widetilde{T}$-values. For $\xi$ function and $\psi$ function, the following theorem is known. Note that, for the empty index $\phi$, we regard $\mm{wt}(\phi)=0$ and the constant $1$ as the value of weight 0.

\begin{theorem}[{\cite{KT2, PX}}]
	For $\Bbbk=(k_1,\ldots,k_r)\in\Z_{\geq1}^r$, we have
	\begin{align*}
	\xi(\Bbbk;s)&=\sum_{{\Bbbk}^{\prime}, j\geq0}{} C_{\Bbbk}({\Bbbk}^{\prime};j)\binom{s+j-1}{j}\zeta({\Bbbk}^{\prime},s+j),\\
	\psi(\Bbbk;s)&=\sum_{{\Bbbk}^{\prime}, j\geq0}{} C^{\prime}_{\Bbbk}({\Bbbk}^{\prime};j)\binom{s+j-1}{j}T({\Bbbk}^{\prime},s+j).\\
	\end{align*}
	Here, the sum runs over all indices $\Bbbk^{\prime}$ and $j\in\Z_{\geq0}$ with $\mm{wt}(\Bbbk^{\prime})+j\leq\mm{wt}(\Bbbk)$, and $C_{\Bbbk}({\Bbbk}^{\prime};j), C^{\prime}_{\Bbbk}({\Bbbk}^{\prime};j)$ are $\Q$-linear combination of multiple zeta values and multiple $T$-values of weight $\mm{wt}(\Bbbk)-\mm{wt}({\Bbbk}^{\prime})-j$, respectively.
\end{theorem}

A similar theorem holds for $\lambda$ function and multiple $\widetilde{T}$-values.

\begin{theorem}\label{theoremexp}
	For $\Bbbk=(k_1,\ldots,k_r)\in\left(\Z_{\geq1}\right)^r$, we have
	\[
	\lambda(\Bbbk;s)=i^{\mm{wt}(\Bbbk)-\mm{dep}(\Bbbk)}\sum_{{\Bbbk}^{\prime}, j\geq0}{} \widetilde{C}_{\Bbbk}({\Bbbk}^{\prime};j)\binom{s+j-1}{j}\widetilde{T}({\Bbbk}^{\prime},s+j).
	\]
	Here, the sum runs over all indices $\Bbbk^{\prime}$ and $j\in\Z_{\geq0}$ with $\mm{wt}(\Bbbk^{\prime})+j\leq\mm{wt}(\Bbbk)$, and $\widetilde{C}_{\Bbbk}({\Bbbk}^{\prime};j)$ is $\Q$-linear combination of multiple $\widetilde{T}$-values of weight $\mm{wt}(\Bbbk)-\mm{wt}({\Bbbk}^{\prime})-j$.
\end{theorem}

To show Theorem \ref{theoremexp}, we need the following lemmas.

\begin{lemma}\label{1apoly}\cite[pp.24]{KT1}
	For $r\in\Z_{\geq1}$, we have
	\[
	\apoly{\{1\}_r}{z}=\frac{1}{r!}\left(\apoly{1}{z}\right)^r.
	\]
\end{lemma}

\begin{lemma}\label{inttilt}
	For $r\in\Z_{\geq2}$, $\Bbbk=(k_1,\ldots,k_{r-1})\in\Z_{\geq1}^r$ and $s\in\Co$ with $\mm{Re}(s)>1$, we have
	\[
	i^{r-1}\widetilde{T}(\Bbbk,s)=\frac{1}{\Gamma(s)}\int_{0}^{\infty} t^{s-1}\frac{\mm{A}(\Bbbk;ie^{-t})}{\cosh{t}} dt.
	\]
\end{lemma}

\begin{proof}
    By definition of multiple $\widetilde{T}$-values and $\frac{1}{\Gamma(s)}\int_{0}^{\infty} t^{s-1}e^{-nt}=\frac{1}{n^s}$, we have
     \begin{align*}
    	\widetilde{T}({\Bbbk}_{r-1},s)
    	&=2^r\sum_{\substack{0<m_1<\cdots<m_{r-1}\\ m_j\equiv j \bmod{2}}}^{} \frac{1}{m_1^{k_1}\cdots m_{r-1}^{k_{r-1}}} \sum_{\substack{m_r=m_{r-1}+1\\ m_r\equiv r\bmod{2}}}^{\infty} 
    	\frac{(-1)^{(m_r-r)/2}}{m_r^s}\\
    	&=\frac{2^r}{\Gamma(s)}\int_{0}^{\infty} t^{s-1} \sum_{\substack{0<m_1<\cdots<m_{r-1}\\ m_j\equiv j \bmod{2}}}^{} \frac{1}{m_1^{k_1}\cdots m_{r-1}^{k_{r-1}}} \sum_{\substack{m_r=m_{r-1}+1\\ m_r\equiv r\bmod{2}}}^{\infty} (-1)^{(m_r-r)/2}e^{-m_rt} dt\\
    	&=\frac{2^ri^{-r}}{\Gamma(s)}\int_{0}^{\infty} t^{s-1} \sum_{\substack{0<m_1<\cdots<m_{r-1}\\ m_j\equiv j \bmod{2}}}^{} \frac{1}{m_1^{k_1}\cdots m_{r-1}^{k_{r-1}}} \sum_{\substack{m_r=m_{r-1}+1\\ m_r\equiv r\bmod{2}}}^{\infty} (ie^{-t})^{m_r} dt\\
    	&=\frac{2^{r-1}i^{-r+1}}{\Gamma(s)}\int_{0}^{\infty} t^{s-1} \sum_{\substack{0<m_1<\cdots<m_{r-1}\\ m_j\equiv j \bmod{2}}}^{} \frac{1}{m_1^{k_1}\cdots m_{r-1}^{k_{r-1}}}\frac{2(ie^{-t})^{m_{r-1}}}{e^t+e^{-t}}dt\\
    	&=\frac{i^{-r+1}}{\Gamma(s)}\int_{0}^{\infty} t^{s-1} \frac{\mm{A}(\Bbbk;ie^{-t})}{\cosh{t}} dt.
    \end{align*}
    Hence we obtain Lemma \ref{inttilt}.
\end{proof}

\begin{lemma}\label{expliapoly}
	For $\Bbbk=(k_1,\ldots,k_r)\in\Z_{\geq1}^r$, we have
	\[
	\apoly{\Bbbk}{\frac{1+z}{1-z}}=\sum_{{\Bbbk}^{\prime}, j\geq0}{} i^{\mm{wt}(\Bbbk)-\mm{dep}(\Bbbk)-\mm{dep}(\Bbbk^{\prime})}\widetilde{C}_{\Bbbk}({\Bbbk}^{\prime};j)\apoly{\{1\}_j}{\frac{1+z}{1-z}}\mm{A}\left({\Bbbk}^{\prime};z\right).
	\]
	Here, the sum runs over all indices $\Bbbk^{\prime}$ and $j\in\Z_{\geq0}$ with $\mm{wt}(\Bbbk^{\prime})+j\leq\mm{wt}(\Bbbk)$, and $\widetilde{C}_{\Bbbk}({\Bbbk}^{\prime};j)$ is the same as Theorem \ref{theoremexp}. Also, we regard $\apoly{\phi}{z}=\mm{A}(\phi;z)=1$.
\end{lemma}

\begin{proof}
	We prove the lemma by induction on $\mm{wt}(\Bbbk)$. The case $\Bbbk=(1)$ is obvious. Suppose that the Lemma holds for indices whose weight are less than $\mm{wt}(\Bbbk)$. We first suppose that  $\Bbbk$ is admissible. We write
	\[
	\apoly{\Bbbk_{-}}{\frac{1+z}{1-z}}=\sum_{({\Bbbk_{-}})^{\prime}, j\geq0}^{} i^{\mm{wt}({\Bbbk_{-}})-\mm{dep}(\Bbbk_{-})-\mm{dep}(({\Bbbk_{-}})^{\prime})}\widetilde{C}_{\Bbbk_{-}}(({\Bbbk_{-}})^{\prime};j)\apoly{\{1\}_j}{\frac{1+z}{1-z}}\mm{A}(({\Bbbk_{-}})^{\prime};z).
	\]
	Then, by integration by parts, we have
		\begin{align*}
			\apoly{\Bbbk}{\frac{1+z}{1-z}}&=\int_{i}^{z} \frac{2}{1-u^2} \apoly{\Bbbk_-}{\frac{1+u}{1-u}}du\\
			&=\sum_{({\Bbbk_{-}})^{\prime}, j\geq0}^{} i^{\mm{wt}({\Bbbk_{-}})-\mm{dep}(\Bbbk_{-})-\mm{dep}(({\Bbbk_{-}})^{\prime})}\widetilde{C}_{\Bbbk_{-}}(({\Bbbk_{-}})^{\prime};j) \\
			&\times\left(\sum_{l=0}^{j}\apoly{\{1\}_{j-l}}{\frac{1+z}{1-z}}\mm{A}(({\Bbbk_{-}})^{\prime},l+1;z)-i^{\mm{dep}((\Bbbk_{-})^{\prime},j+1)}\widetilde{T}(({\Bbbk_{-}})^{\prime},j+1)\right).
		\end{align*}
	Since $\mm{wt}({\Bbbk_{-}})-\mm{dep}(\Bbbk_{-})-\mm{dep}(({\Bbbk_{-}})^{\prime})=\mm{wt}(\Bbbk)-\mm{dep}(\Bbbk)-\mm{dep}(({\Bbbk_{-}})^{\prime},l+1)$, Lemma \ref{expliapoly} holds.
    Suppose that $\Bbbk$ is non-admissible. We write $\Bbbk_r=\Bbbk$ and $\Bbbk_{r-1}=(k_1,\ldots,k_{r-1})$. In this case, by regularization of polylogarithms and Lemma \ref{1apoly}, we have 
    \begin{align*}
    	\apoly{\Bbbk_r}{\frac{1+z}{1-z}}&=\apoly{\Bbbk_{r-1}}{\frac{1+z}{1-z}}\apoly{1}{\frac{1+z}{1-z}}\\
    	&\five-\sum_{l=1}^{r-1}\sum_{\substack{k_{l,1}+k_{l,2}=k_l+1\\k_{l,1},k_{l,2}\geq1, k_{1,1}\geq2}}^{}\apoly{k_1,\ldots,k_{l-1},k_{l,1},k_{l,2},k_{l+1},\ldots,k_{r-1}}{\frac{1+z}{1-z}}.
    \end{align*}
    Let $\widetilde{\Bbbk}_{r-1,l}=(k_1,\ldots,k_{l,1},k_{l,2},\ldots,k_{r-1})$. By induction hypothesis, we have
     \begin{align*}
    	\apoly{\Bbbk_r}{\frac{1+z}{1-z}}=&\sum_{(\Bbbk_{r-1})^{\prime}, j\geq0}^{} i^{\mm{wt}(\Bbbk_{r-1})-\mm{dep}(\Bbbk_{r-1})-\mm{dep}(({\Bbbk_{r-1}})^{\prime})}\widetilde{C}_{\Bbbk_{r-1}}((\Bbbk_{r-1})^{\prime};j)(j+1)\\
    	&\times\apoly{\{1\}_{j+1}}{\frac{1+z}{1-z}}\mm{A}((\Bbbk_{r-1})^{\prime};z)\\
    	&-\sum_{l=1}^{r-1}\sum_{\substack{k_{l,1}+k_{l,2}=k_l+1\\k_{l,1},k_{l,2}\geq1, k_{1,1}\geq2}}^{}\sum_{(\widetilde{\Bbbk}_{r-1,i})^{\prime}, j_1\geq0}^{} i^{\mm{wt}(\widetilde{\Bbbk}_{r-1,l})-
    		\mm{dep}(\widetilde{\Bbbk}_{r-1,l})-\mm{dep}((\widetilde{\Bbbk}_{r-1,l})^{\prime})}\widetilde{C}_{\widetilde{\Bbbk}_{r-1,l}}({(\widetilde{\Bbbk}_{r-1,l})}^{\prime};j_1)\\
    	&\times\apoly{\{1\}_{j_1}}{\frac{1+z}{1-z}}\mm{A}\left((\widetilde{\Bbbk}_{r-1,l}^{\prime});z\right).
    \end{align*}
    Since
    \begin{align*}
    	\mm{wt}(\Bbbk_{r-1})-\mm{dep}(\Bbbk_{r-1})-\mm{dep}(({\Bbbk_{r-1}})^{\prime})&=\mm{wt}(\Bbbk_r)-\mm{dep}(\Bbbk_r)-\mm{dep}(({\Bbbk_{r-1}})^{\prime}),\\
    	\mm{wt}(\widetilde{\Bbbk}_{r-1,l})-\mm{dep}(\widetilde{\Bbbk}_{r-1,l})-\mm{dep}((\widetilde{\Bbbk}_{r-1,l})^{\prime})&=\mm{wt}(\Bbbk_r)-\mm{dep}(\Bbbk_r)-\mm{dep}((\widetilde{\Bbbk}_{r-1,l})^{\prime}),
    \end{align*}
    we obtain Lemma \ref{expliapoly}.
\end{proof}

\begin{proof}[Proof of Theorem \ref{theoremexp}]
	We put $z=ie^{-t}$ in Lemma \ref{expliapoly}. By \eqref{apolya}, we have 
	\begin{align}\label{dai1}
		\apoly{\Bbbk}{\tanh^{-1}{\left(\frac{t}{2}+\frac{\pi i}{4}\right)}}=\sum_{{\Bbbk}^{\prime}, j\geq0}{} i^{\mm{wt}(\Bbbk)-\mm{dep}(\Bbbk)-\mm{dep}(\Bbbk^{\prime})}\widetilde{C}_{\Bbbk}({\Bbbk}^{\prime};j)\frac{t^j}{j!}\mm{A}\left({\Bbbk}^{\prime};ie^{-t}\right).
	\end{align}
    Hence we obtain Theorem \ref{theoremexp} by substituting \eqref{dai1} into \eqref{lambda} and applying Lemma \ref{inttilt}.
\end{proof}

\begin{example}
	For $\Bbbk=(2)$, a direct calculation shows that
	\begin{align*}
		\apoly{2}{\frac{1+z}{1-z}}=\apoly{1}{\frac{1+z}{1-z}}\mm{A}(1;z)+\mm{A}(2;z)-i\widetilde{T}(2).
	\end{align*}
    Hence we have
    \begin{align*}
    	\lambda(2;s)&=is\widetilde{T}(1,s+1)+i\widetilde{T}(2,s)-i\widetilde{T}(2)\widetilde{T}(s).
    \end{align*}
    By putting $s=3$ in the above equation, then we obtain
    \begin{align*}
    	\lambda(2;3)=3i\widetilde{T}(1,4)+i\widetilde{T}(2,3)-i\widetilde{T}(2)\widetilde{T}(3).
    \end{align*}
    On the other hand, by \cite[Theorem 4.3]{KT2}, we have
    \[
    \lambda(2;3)=-i\left(3\widetilde{T}(1,4)+2\widetilde{T}(2,3)+\widetilde{T}(3,2)\right).
    \]
    Hence we obtain
    \[
    \widetilde{T}(2)\widetilde{T}(3)=6\widetilde{T}(1,4)+3\widetilde{T}(2,3)+\widetilde{T}(3,2)
    \]
    (this is also obtained by the shuffle product).
\end{example}

Furthermore, we can write some values of $\lambda$ function for certain indices explicitly. For $\xi$ and $\lambda$ functions, the following theorem is known.

\begin{theorem}[{\cite[Theorem 2.15, 2.16]{PX}}]
	For $j, r\in\Z_{\geq1}$ with $j\leq r$ and $s\in\Co$ with $\mm{Re}(s)>1$, we have
	\begin{align*}
		\xi(\{1\}_{j-1},2,\{1\}_{r-j};s)&=(-1)^{r-j-1}\sum_{m=r-j}^{r}\binom{m}{r-j}\binom{s+r-m-2}{r-m-1}\zeta(m+2,s+r-m-1)\\
		&+\sum_{l=0}^{r-j}(-1)^l\binom{j+l}{l}\binom{s+r-j-l-1}{r-j-l}\zeta(s+r-j-l)\zeta(j+l+1),\\
		\psi(\{1\}_{j-1},2,\{1\}_{r-j};s)&=(-1)^{r-j-1}\sum_{m=r-j}^{r}\binom{m}{r-j}\binom{s+r-m-1}{r-m}T(m+1,s+r-m)\\
		&+\sum_{l=0}^{r-j}(-1)^l\binom{j+l}{l}\binom{s+r-j-l-1}{r-j-l}\widetilde{T}(s+r-j-l)\widetilde{T}(j+l+1).
	\end{align*}
\end{theorem}

To obtain an analogous relation for $\lambda$ function, we need the following lemma.

\begin{lemma}\label{lemma211}
	For $m, n\in\Z_{\geq1}$, we have
	\begin{align*}
		\int_{i}^{z} \apoly{1}{u}^m\mm{A}\left(1;\frac{1+u}{1-u}\right)^n\frac{2du}{1-u^2}&=n!\sum_{l=0}^{m}\frac{m!}{(m-l)!}(-1)^l\apoly{1}{z}^{m-l}\mm{A}(\{1\}_{n-1},l+2;\frac{1+z}{1-z})\\
		&\five-n!m!(-1)^mi\widetilde{T}(\{1\}_{n-1},m+2).    
	\end{align*}
\end{lemma}

\begin{proof}
	Since 
	\begin{align*}
		\dv{}{z} \mm{A}\left(k_1,\ldots,k_{r-1},k_r;\frac{1+z}{1-z}\right)=\begin{cases}
			\frac{2}{1-z^2}\mm{A}\left(k_1,\ldots,k_{r-1},k_r-1;\frac{1+z}{1-z}\right) &(\Bbbk_r:\mbox{admissible}),\\
			-\frac{1}{z}\mm{A}\left(k_1,\ldots,k_{r-1};\frac{1+z}{1-z}\right) &(\Bbbk_r:\mbox{non-admissible}),
		\end{cases}
	\end{align*}
    and $\frac{1+i}{1-i}=i$, we have
	\begin{align*}
		\int_{i}^{z} \apoly{1}{u}^m\mm{A}\left(1;\frac{1+u}{1-u}\right)^n\frac{2du}{1-u^2}&=n!\int_{i}^{z} \apoly{1}{u}^m\mm{A}\left(\{1\}_{n};\frac{1+u}{1-u}\right)\frac{2du}{1-u^2}\\
		&=n!\left[\apoly{1}{u}^m\mm{A}\left(\{1\}_{n-1},2;\frac{1+u}{1-u}\right)\right]_{i}^{z} \\
		&\ten-n!\int_{i}^{z} m\apoly{1}{u}^{m-1}\mm{A}\left(\{1\}_{n-1},2;\frac{1+u}{1-u}\right)\frac{2du}{1-u^2}\\
		&=\cdots\\
		&=n!\sum_{l=0}^{m-1} \frac{m!}{(m-l)!}(-1)^l\apoly{1}{z}^{m-l}\mm{A}(\{1\}_{n-1},l+2;\frac{1+z}{1-z})\\
		&\ten+\left[n!m!(-1)^m\mm{A}\left(\{1\}_{n-1},m+2;\frac{1+u}{1-u}\right)\right]_{i}^{z}\\
		&=n!\sum_{l=0}^{m} \frac{m!}{(m-l)!}(-1)^l\apoly{1}{z}^{m-l}\mm{A}(\{1\}_{n-1},l+2;\frac{1+z}{1-z})\\
		&\ten-n!m!(-1)^mi\widetilde{T}(\{1\}_{n-1},m+2).
	\end{align*}
    Hence we obtain Lemma \ref{lemma211}.
\end{proof}

\begin{proposition}\label{12lam}
	For $r, j\in\Z_{\geq1}$ with $r>j$ and $s\in\Co$ with $\mm{Re}(s)>1$, we have
	\begin{align*}
		\lambda(\{1\}_{j-1},2,\{1\}_{r-j};s)&=i(-1)^{r-j}\sum_{m=r-j}^{r}\binom{m}{r-j}\binom{s+r-m-1}{r-m}\widetilde{T}(m+1,s+r-m)\\
		&-i\sum_{l=0}^{r-j}(-1)^l\binom{j+l}{l}\binom{s+r-j-l-1}{r-j-l}\widetilde{T}(s+r-j-l)\widetilde{T}(j+l+1).
	\end{align*}
\end{proposition}

\begin{proof}
	For any index $\Bbbk_r=(k_1,\ldots,k_r)$, since
	\begin{align}\label{a11}
		\int_{a<u_1<\cdots<u_r<b}^{} \frac{2du_1}{1-u_1^2}\cdots\frac{2du_r}{1-u_r^2} =\frac{1}{r!}\left(\apoly{1}{b}-\apoly{1}{a}\right)^r,
	\end{align}
	we have
	\begin{align}
	\begin{aligned}\label{aitel}
		\apoly{\Bbbk_r}{z}=&\int_{i<t_1<\cdots<t_{k_r}<z}^{} \frac{2dt_1}{1-t_1^2}\underbrace{\frac{dt}{t}\cdots\frac{dt}{t}}_{k_1-1}\cdots\frac{2dt_{k_{r-1}+1}}{1-t_{k_{r-1}+1}^2}\underbrace{\frac{dt}{t}\cdots\frac{dt}{t}}_{k_r-1}\\
		=&(-1)^r\int_{i<t_1<\cdots<t_{k_r}<\frac{z-1}{z+1}}^{} \frac{dt_1}{t_1}\underbrace{\frac{2dt}{1-t^2}\cdots\frac{2dt}{1-t^2}}_{k_1-1}\cdots\frac{dt_{k_{r-1}+1}}{t_{k_{r-1}+1}}\underbrace{\frac{2dt}{1-t^2}\cdots\frac{2dt}{1-t^2}}_{k_r-1}\\
		=&\frac{(-1)^r}{\prod_{l=1}^{r}(k_l-1)!}\int_{i<t_1<\cdots<t_{r}<\frac{z-1}{z+1}}{} \frac{dt_1}{t_1}\left(\apoly{1}{t_2}-\apoly{1}{t_1}\right)^{k_1-1}\cdots\\
		&\times\frac{dt_r}{t_r}\left(\apoly{1}{\frac{z-1}{z+1}}-\apoly{1}{t_r}\right)^{k_r-1}.
	\end{aligned}
    \end{align}
    Here, we put $\Bbbk_r=(\{1\}_{j-1},2,\{1\}_{r-j})$ in the above equation. A direct calculation shows that
	\begin{align*}
		\apoly{\Bbbk_r}{\frac{1+z}{1-z}}&=(-1)^r\int_{i<t_1<\cdots<t_{r}<z}{} \frac{\apoly{1}{t_{j+1}}-\apoly{1}{t_j}}{t_1\cdots t_r}dt_1\cdots dt_r\\
		&=(-1)^r\int_{i<t_1<\cdots<t_{r}<\frac{z-1}{z+1}}{} \frac{\apoly{1}{\frac{1+u_{j+1}}{1-u_{j+1}}}-\apoly{1}{\frac{1+u_j}{1-u_j}}}{\frac{1}{2}(1-u_1^2)\cdots \frac{1}{2}(1-u_r^2)}du_1\cdots du_r\\
		&=(-1)^r\int_{i}^{\frac{z-1}{z+1}}\frac{\apoly{1}{u}^j}{j!(r-j-1)!}\left(\apoly{1}{\frac{z-1}{z+1}}-\apoly{1}{u}\right)^{r-j-1}\apoly{1}{\frac{1+u}{1-u}}\frac{2du}{1-u^2}\\
		&\five-(-1)^r\int_{i}^{\frac{z-1}{z+1}}\frac{\apoly{1}{u}^{j-1}}{(j-1)!(r-j)!}\left(\apoly{1}{\frac{z-1}{z+1}}-\apoly{1}{u}\right)^{r-j}\apoly{1}{\frac{1+u}{1-u}}\frac{2du}{1-u^2}\\
		&=\sum_{l=0}^{r-j-1}\binom{r-1-j}{l}\frac{(-1)^{r+l}}{(r-j-1)!j!}\apoly{1}{\frac{z-1}{z+1}}^{r-1-j-l}\int_{i}^{\frac{z+1}{z-1}}\apoly{1}{u}^{j+l}\apoly{1}{\frac{1+u}{1-u}}\frac{2du}{1-u^2}\\
		&\five-\sum_{l=0}^{r-j}\binom{r-j}{l}\frac{(-1)^{r+l}}{(r-j)!(j-1)!}\apoly{1}{\frac{z-1}{z+1}}^{r-j-l}\int_{i}^{\frac{z+1}{z-1}}\apoly{1}{u}^{j+l-1}\apoly{1}{\frac{1+u}{1-u}}\frac{2du}{1-u^2}\\
		&=\sum_{l=0}^{r-j}\frac{(-1)^{r+l-1}(j+l)}{(r-j-l)!j!l!}\apoly{1}{\frac{z-1}{z+1}}^{r-j-l}\int_{i}^{\frac{z+1}{z-1}}\apoly{1}{u}^{j+l-1}\apoly{1}{\frac{1+u}{1-u}}\frac{2du}{1-u^2}\\
		&=\sum_{l=0}^{r-j}\frac{(-1)^{r+l-1}(j+l)}{(r-j-l)!j!l!}\apoly{1}{\frac{z-1}{z+1}}^{r-j-l}\int_{i}^{\frac{z+1}{z-1}}\apoly{1}{u}^{j+l-1}\mm{A}\left(1;\frac{1+u}{1-u}\right)\frac{2du}{1-u^2}\\
		&\five-\frac{\pi i}{2}\sum_{l=0}^{r-j}\frac{(-1)^{r+l-1}(j+l)}{(r-j-l)!j!l!}\apoly{1}{\frac{z-1}{z+1}}^{r-j-l}\int_{i}^{\frac{z+1}{z-1}}\apoly{1}{u}^{j+l-1}\frac{2du}{1-u^2}\\
		&=\sum_{l=0}^{r-j}\sum_{m=0}^{j+l-1}\frac{(-1)^{r+l-1-m}(j+l)!}{(r-j-l)!(j+l-1-m)!j!l!}\apoly{1}{\frac{z-1}{z+1}}^{r-1-m}\mm{A}(m+2;z)\\
		&\five-i\sum_{l=0}^{r-j}\frac{(-1)^{r+j}(j+l)!}{(r-j-l)!j!l!}\apoly{1}{\frac{z-1}{z+1}}^{r-j-l}\widetilde{T}(j+l+1)\\
		&\five-\frac{\pi i}{2}\sum_{l=0}^{r-j}\frac{(-1)^{r+l-1}}{(r-j-l)!j!l!}\apoly{1}{\frac{z-1}{z+1}}^{r}.
	\end{align*}
	Since $\frac{ie^{-t}-1}{ie^{-t}+1}=\tanh{\left(-\frac{t}{2}+\frac{\pi i}{4}\right)}$, we obtain
	\begin{align*}
		\apoly{\Bbbk_r}{\tanh^{-1}{\left(\frac{t}{2}+\frac{\pi i}{4}\right)}}&=\sum_{l=0}^{r-j}\sum_{m=0}^{j+l-1}\frac{(-1)^{l}(j+l)!}{(r-j-l)!(j+l-1-m)!j!l!}t^{r-1-m}\mm{A}(m+2;ie^{-t})\\
		&\five-i\sum_{l=0}^{r-j}\frac{(-1)^{l}(j+l)!}{(r-j-l)!j!l!}t^{r-j-l}\widetilde{T}(j+l+1)-i\widetilde{T}(1)\sum_{l=0}^{r-j}\frac{(-1)^{l+1}}{(r-j-l)!j!l!}t^r
	\end{align*}
	and
	\begin{align*}
		\lambda(\{1\}_{j-1},2,\{1\}_{r-j};s)&=i\sum_{l=0}^{r-j}\sum_{m=0}^{j+l-1}\frac{(-1)^l(r-m-1)!(j+l)!}{(r-j-l)!(j+l-m-1)!}\binom{s+r-m-2}{r-m-1}\widetilde{T}(m+2,s+r-m-1)\\
		&-i\sum_{l=0}^{r-j}\frac{(-1)^l(j+l)!}{j!l!}\binom{s+r-j-l-1}{r-j-l}\widetilde{T}(s+r-j-l)\widetilde{T}(j+l+1)\\
		&-i\widetilde{T}(1)\sum_{l=0}^{r-j}\frac{(-1)^{l+1}r!}{(r-j-l)!j!l!}\binom{s+r-1}{r}\widetilde{T}(s+r).
	\end{align*}
	For the first sum, in a similar way to \cite[Theorem 2.8]{PX}, we have
	\begin{align*}
		&\sum_{l=0}^{r-j}\sum_{m=0}^{j+l-1}\frac{(-1)^l(r-m-1)!(j+l)!}{(r-j-l)!(j+l-m-1)!}\binom{s+r-m-2}{r-m-1}\widetilde{T}(m+2,s+r-m-1)\\
		&=\sum_{m=1}^{r}\binom{s+r-m-1}{r-m}\widetilde{T}(m+1,s+r-m)\sum_{l=0}^{r-j}(-1)^l\binom{r-m-1}{r-j-l}\binom{j+l}{j}\\
		&=(-1)^{r-j}\sum_{m=r-j}^{r}\binom{m}{r-j}\binom{s+r-m-1}{r-m}\widetilde{T}(m+1,s+r-m).
	\end{align*}
	For the last sum, we have
	\begin{align*}
		\sum_{l=0}^{r-j}\frac{(-1)^{l+1}r!}{(r-j-l)!j!l!}\binom{s+r-1}{r}\widetilde{T}(s+r)&=-\binom{r}{j}\binom{s+r-1}{r}\widetilde{T}(s+r)\sum_{l=0}^{r-1}(-1)^l\binom{r+j}{l}\\
		&=0.
	\end{align*} 
    Hence we obtain Proposition \ref{12lam}.
\end{proof}

\begin{example}
	For $r=2$ and $j=1$, we have
	\[
	\lambda(2,1;s)=-is\widetilde{T}(2,s+1)-2i\widetilde{T}(3,s)-is\widetilde{T}(2)\widetilde{T}(s+1)+2i\widetilde{T}(3)\widetilde{T}(s).
	\]
	We also obtain the above formula by Theorem \ref{theoremexp}.
\end{example}

For $r=j$, we have the following formula.

\begin{proposition}
	For $r\in\Z_{\geq1}$ and $s\in\Co$ with $\mm{Re}(s)>1$, we have
	\[
	\lambda(\{1\}_{r-1},2;s)=i\binom{s+r-1}{r}\widetilde{T}(1,s+r)+i\sum_{l=0}^{r-1}\binom{s+r-l-2}{r+l-1}\widetilde{T}(l+2,s+r-l-1)-i\widetilde{T}(r+1)
	\widetilde{T}(s).
	\]
\end{proposition}

\begin{proof}
	We put $\Bbbk_r=(\{1\}_{r-1},2)$ in \eqref{aitel}. As similar way to Proposition \ref{12lam}, we have
	\begin{align*}
		\apoly{\Bbbk_r}{\frac{1+z}{1-z}}&=(-1)^r\int_{i<t_1<\cdots<t_{r}<\frac{z-1}{z+1}}{} \frac{\apoly{1}{z}-\apoly{1}{\frac{1+u_r}{1-u_r}}}{\frac{1}{2}(1-u_1^2)\cdots \frac{1}{2}(1-u_r^2)}du_1\cdots du_r\\
		&=\frac{(-1)^r}{r!}\left(\mm{A}(1;z)-i\widetilde{T}(1)\right)\apoly{1}{\frac{z-1}{z+1}}^r+\frac{(-1)^r}{r!}\frac{\pi i}{2}\apoly{1}{\frac{z-1}{z+1}}^r\\
		&\five-\frac{(-1)^r}{(r-1)!}\sum_{l=0}^{r-1}\frac{(-1)^l(r-1)!}{(r-1-l)!}\apoly{1}{\frac{z-1}{z+1}}^{r-1-l}\mm{A}(l+2;z)-i\widetilde{T}(r+1)\\
		&=\frac{(-1)^r}{r!}\mm{A}(1;z)\apoly{1}{\frac{z-1}{z+1}}^r-i\widetilde{T}(r+1)\\
		&-\frac{(-1)^r}{(r-1)!}\sum_{l=0}^{r-1}\frac{(-1)^l(r-1)!}{(r-1-l)!}\apoly{1}{\frac{z-1}{z+1}}^{r-1-l}\mm{A}(l+2;z).
	\end{align*}
	Hence we obtain
	\begin{align*}
		\apoly{\Bbbk_r}{\tanh^{-1}{\left(\frac{t}{2}+\frac{\pi i}{4}\right)}}&=\frac{1}{r!}\mm{A}(1;ie^{-t})+\sum_{l=0}^{r-1}\frac{1}{(r-l-1)!}t^{r-l-1}\mm{A}(l+2;ie^{-t})-i\widetilde{T}(r+1),
	\end{align*}
	and the result follows from \eqref{lambda} and Lemma \ref{inttilt}.
\end{proof}

\begin{example}
	For $r=1$, we have
	\[
	\lambda(2;s)=is\widetilde{T}(1,s+1)+i\widetilde{T}(2,s)-i\widetilde{T}(2)\widetilde{T}(s),
	\]
	and we obtain
	\[
	\lambda(2;2)=2i\widetilde{T}(1,3)+i\widetilde{T}(2,2)-i\widetilde{T}(2)^2.
	\]
	On the other hand, by \cite[Theorem 4.3]{KT1}, we obtain
	\[
	\lambda(2;2)=-i\widetilde{T}(2,2)-2i\widetilde{T}(1,3).
	\]
	Hence we have
	\[
	\widetilde{T}(2)^2=4\widetilde{T}(1,3)+2\widetilde{T}(2,2)
	\]
	(this is also obtained by the shuffle product).
\end{example}

Also, Pallewatta and Xu showed the following relations.

\begin{theorem}[{\cite[Theorem 2.17]{PX}}]
	For $k, p$ and $r\in\Z_{\geq1}$, we have
	\begin{align*}
		\sum_{\substack{\mm{wt}(\Bbbk)=k+r-1\\ \mm{dep}(\Bbbk)=r}}^{} \xi(\Bbbk;p)&=\sum_{j=0}^{r-1} (-1)^j\binom{p+r-j-2}{r-j-1} \xi(\{1\}_j,k;p+r-j-1),\\
		\xi(\{1\}_{r-1},k;p)&=\sum_{j=0}^{r-1} \sum_{\substack{\mm{wt}(\Bbbk)=k+j\\ \mm{dep}(\Bbbk)=j+1}}^{} (-1)^j\binom{p+r-j-2}{r-j-1}\xi(\Bbbk;p+r-j-1),\\
    	\sum_{\substack{\mm{wt}(\Bbbk)=k+r-1\\ \mm{dep}(\Bbbk)=r}}^{} \psi(\Bbbk;p)&=\sum_{j=0}^{r-1} (-1)^j\binom{p+r-j-2}{r-j-1} \psi(\{1\}_j,k;p+r-j-1),\\
    	\psi(\{1\}_{r-1},k;p)&=\sum_{j=0}^{r-1} \sum_{\substack{\mm{wt}(\Bbbk)=k+j\\ \mm{dep}(\Bbbk)=j+1}}^{} (-1)^j\binom{p+r-j-2}{r-j-1}\psi(\Bbbk;p+r-j-1).
    \end{align*}
\end{theorem}

We obtain the similar relation for $\lambda$ function.

\begin{proposition}
	For $r, k\in\Z_{\geq1}$ and $s\in\Co$ with $\mm{Re}(s)>1$, we have
	\begin{align}
		\sum_{\substack{\mm{wt}(\Bbbk)=k+r-1\\ \mm{dep}(\Bbbk)=r}}^{} \lambda(\Bbbk;s)&=\sum_{j=0}^{r-1} (-1)^j\binom{s+r-j-2}{r-j-1} \lambda(\{1\}_j,k;s+r-j-1),\label{me1}\\
		\lambda(\{1\}_{r-1},k;s)&=\sum_{j=0}^{r-1} \sum_{\substack{\mm{wt}(\Bbbk)=k+j\\ \mm{dep}(\Bbbk)=j+1}}^{} (-1)^j\binom{s+r-j-2}{r-j-1}\lambda(\Bbbk;s+r-j-1)\label{me2}.
	\end{align}
\end{proposition}

\begin{proof}
    For \eqref{me1}, by the multinomial theorem and \eqref{a11}, we have
	\begin{align*}
		\sum_{\substack{\mm{wt}(\Bbbk)=k+r-1\\ \mm{dep}(\Bbbk)=r}}^{} \apoly{\Bbbk}{z}&=\sum_{\substack{\mm{wt}(\Bbbk)=k+r-1\\ \mm{dep}(\Bbbk)=r}}^{}\int_{i<t_1<\cdots<t_r<z}^{} \prod_{j=1}^{r}\frac{2dt_j}{1-t_j^2}\underbrace{\frac{dt}{t}\cdots\frac{dt}{t}}_{k_j-1}\\
		&=\frac{(-1)^r}{(k-1)!}\int_{i<t_1<\cdots<t_{r}<\frac{z-1}{z+1}}{}\frac{\left(\apoly{1}{\frac{z-1}{z+1}}-\apoly{1}{t_r}\right)^{k-1}}{t_1\cdots t_r}dt_1\cdots dt_r\\
		&=-\frac{1}{(r-1)!(k-1)!}\int_{i}^{\frac{z-1}{z+1}}\frac{\left(\apoly{1}{\frac{z-1}{z+1}}-\apoly{1}{t}\right)^{k-1}}{t}\left(\apoly{1}{z}-\apoly{1}{\frac{1+t}{1-t}}\right)^{r-1}dt\\
		&=-\sum_{j=0}^{r-1}\binom{r-1}{j}\frac{(-1)^j\apoly{1}{z}^{r-j-1}}{(r-1)!(k-1)!}\int_{i}^{\frac{z-1}{z+1}}\frac{\apoly{1}{\frac{1+t}{1-t}}^j\left(\apoly{1}{\frac{z-1}{z+1}}-\apoly{1}{t}\right)^{k-1}}{t}dt\\
		&=-\frac{1}{(k-1)!}\sum_{j=0}^{r-1}\frac{\apoly{1}{z}^{r-j-1}}{(r-j-1)!}\int_{i<t_1<\cdots<t_{j}<t<\frac{z-1}{z+1}}{}\frac{\left(\apoly{1}{\frac{z-1}{z+1}}-\apoly{1}{t}\right)^{k-1}}{t_1\cdots t_jt}dt_1\cdots dt_jdt\\
		&=\sum_{j=0}^{r-1}\frac{(-1)^j}{(r-j-1)!}\apoly{1}{z}^{r-j-1}\apoly{\{1\}_j,k}{z}.
	\end{align*}
	By putting $z=\tanh{\left(\frac{t}{2}+\frac{\pi}{4}i\right)}$ and substituting this to \eqref{lambda}, we obtain \eqref{me1}. For \eqref{me2}, we have
	\begin{align*}
		\sum_{j=0}^{r-1}\frac{(-1)^j\apoly{1}{z}^{r-1-j}}{(r-j-1)!}\sum_{\substack{\mm{wt}(\Bbbk)=k+j\\ \mm{dep}(\Bbbk)=j+1}}^{} \apoly{\Bbbk}{z}&=\sum_{j=0}^{r-1}\frac{(-1)^j\apoly{1}{z}^{r-1-j}}{(r-j-1)!}
		\sum_{l=0}^{j}\frac{(-1)^l\apoly{1}{z}^{j-l}}{(j-l)!}\apoly{\{1\}_l,k}{z}\\
		&=\sum_{l=0}^{r-1}(-1)^l\apoly{1}{z}^{r-l-1}\apoly{\{1\}_l,k}{z}\sum_{j=l}^{r-1}\frac{(-1)^j}{(r-j-1)!(j-l)!}\\
		&=\sum_{l=0}^{r-1}\frac{(-1)^l}{(r-l-1)!}\apoly{1}{z}^{r-l-1}\apoly{\{1\}_l,k}{z}\sum_{j=l}^{r-1}(-1)^j\binom{r-l-1}{j-l}.\\
	\end{align*}
	Since the last sum is 0 except $l=r-1$, we have
	\[
	\sum_{j=0}^{r-1}\frac{(-1)^j}{(r-j-1)!}\apoly{1}{z}^{r-1-j}\sum_{\substack{\mm{wt}(\Bbbk)=k+j\\ \mm{dep}(\Bbbk)=j+1}}^{} \apoly{\Bbbk}{z}=\apoly{\{1\}_{r-1},k}{z}.
	\]
	By putting $z=\tanh{\left(\frac{t}{2}+\frac{\pi}{4}i\right)}$ and substituting this to \eqref{lambda}, we obtain \eqref{me2}.
\end{proof}

\begin{example}
	For $r=k=2$, we have 
	\begin{align*}
		(LHS~\mbox{of}~(\ref{me1}))&=\lambda(1,2;s)+\lambda(2,1;s)\\
		&=i\binom{s+1}{2}\widetilde{T}(1,s+2)-i\widetilde{T}(3,s)-is\widetilde{T}(2)\widetilde{T}(s+1)+i\widetilde{T}(3)\widetilde{T}(s),\\
		(RHS~\mbox{of}~(\ref{me1}))&=s\lambda(2;s+1)-\lambda(1,2;s)\\
		&=i\binom{s+1}{2}\widetilde{T}(1,s+2)-is\widetilde{T}(2)\widetilde{T}(s+1)-i\widetilde{T}(3,s)+i\widetilde{T}(3)\widetilde{T}(s).
	\end{align*}
	Hence we can see that \eqref{me1} holds. Also, by the above equations, we have
	\[
	\lambda(1,2;s)=s\lambda(2;s+1)-\lambda(1,2;s)-\lambda(2,1;s).
	\]
	Hence we can see that \eqref{me2} holds.
\end{example}

\section{Duality type relation}

In this section, we show that $\lambda$ function satisfies a kind of duality relation. To state the theorem, we prepare some notations used in \cite{KY}. For an index $\Bbbk_r=(k_1,\ldots,k_r)\in\Z_{\geq0}^r$ and $j\in\Z_{\geq0}$, we use the following notations:
\begin{align*}
	\overrightarrow{\Bbbk_j}&=(k_1,k_2,\ldots,k_j),~~~\overleftarrow{\Bbbk_j}=(k_r,k_{r-1},\ldots,k_{r+1-j}),~~~\overrightarrow{\Bbbk_0}=\overleftarrow{\Bbbk_0}=\phi.
\end{align*}

Also, for two indices $(k_1,\ldots, k_r)\in\Z_{\geq1}^r$ and $(l_1,\ldots,l_s)\in\Z_{\geq1}^s$, Kaneko and Yamamoto introduced the `circled harmonic product' in \cite{KY} as follows:
\begin{align}\label{chs}
\zeta\left((k_1,\ldots,k_r)~\textcircled{$\ast$}~(l_1,\ldots,l_s)\right)=\zeta\left((k_1,\ldots,k_{r-1})\ast(l_1,\ldots,l_{s-1}),k_r+l_s\right).
\end{align}

By using a level two analogue of \eqref{chs}, Pallewatta and Xu shows the following relations.

\begin{theorem}[{\cite[Theorem 3.4]{PX},\cite[Theorem 3.3]{Xu}}]
	For $p, q, r\in\Z_{\geq1}$ and $\Bbbk_r=(k_1,\ldots,k_r)\in\Z_{\geq2}^r$, we have
    \begin{align*}
    	\xi&(\{1\}_{q-1},(\overrightarrow{\Bbbk_{r}})_{-};p+1)-(-1)^{\mm{wt}(\Bbbk_r)}\xi(\{1\}_{p-1},(\overleftarrow{\Bbbk_{r}})_{-};q+1)\\
    	&=\sum_{j=0}^{r-1}(-1)^{\mm{wt}(\overleftarrow{\Bbbk_{j}})}\sum_{i=1}^{k_{r-j}-2}(-1)^{i-1}\zeta(\{1\}_{p-1},\overleftarrow{\Bbbk_{j}},i+1)\zeta(\{1\}_{q-1},\overrightarrow{\Bbbk_{r-j-1}},k_{r-j}-i)\\
    	&+\sum_{j=0}^{r-2}(-1)^{\mm{wt}(\overleftarrow{\Bbbk_{j+1}})}\left\{\zeta(\{1\}_{q-1},\overrightarrow{\Bbbk_{r-j-1}})\xi(\{1\}_{p-1},(\overleftarrow{\Bbbk_{j+1}})_{-};2)-\zeta(\{1\}_{p-1},\overleftarrow{\Bbbk_{j+1}})\xi(\{1\}_{q-1},(\overrightarrow{\Bbbk_{r-1-j}})_{-};2)\right\},\\
    	\psi&(\{1\}_{q-1},(\overrightarrow{\Bbbk_{r}})_{-};p+1)-(-1)^{\mm{wt}(\Bbbk_r)}\psi(\{1\}_{p-1},(\overleftarrow{\Bbbk_{r}})_{-};q+1)\\
    	&=\sum_{j=0}^{r-1}(-1)^{\mm{wt}(\overleftarrow{\Bbbk_{j}})}\sum_{i=1}^{k_{r-j}-2}(-1)^{i-1}T(\{1\}_{p-1},\overleftarrow{\Bbbk_{j}},i+1)T(\{1\}_{q-1},\overrightarrow{\Bbbk_{r-j-1}},k_{r-j}-i)\\
    	&+\sum_{j=0}^{r-2}(-1)^{\mm{wt}(\overleftarrow{\Bbbk_{j+1}})}\left\{T(\{1\}_{q-1},\overrightarrow{\Bbbk_{r-j-1}})T\left(\left(\{1\}_{p-1},\overleftarrow{\Bbbk_{j+1}}\right)_{-}\textup{\textcircled{$\ast$}}~(1,1)\right)\right.\\
    	&\twen~~~~-T(\{1\}_{p-1},\overleftarrow{\Bbbk_{j+1}})T\left(\left(\{1\}_{q-1},\overrightarrow{\Bbbk_{r-j-1}}\right)_{-}\textup{\textcircled{$\ast$}}~(1,1)\right)\\
    	&\left.\twen~~~+2\delta_{p+j,q+r-j-2}\log{(2)}T(\{1\}_{p-1},\overleftarrow{\Bbbk_{j+1}})T(\{1\}_{q-1},\overrightarrow{\Bbbk_{r-j-1}})\right\}.
    \end{align*}
    Here, $T(\Bbbk~\textup{\textcircled{$\ast$}}~(1,1))$ is a level two analogue of \eqref{chs} ( called the convoluted $T$-values ) defined in \cite{CZ}, and $\delta_{r,s}$ is defined by
    \begin{align*}
    	\delta_{r,s}=\begin{cases}
    		0~&(r\equiv s\bmod{2}),\\
    		-1~&(r:even, s:odd),\\
    		1~&(r:odd, s:even)
    	\end{cases}
    \end{align*}
    for for $r,s\in\Z_{\geq1}$. $($For details of convoluted $T$-values, see \cite{PX,CZ}.$)$
\end{theorem}

The proof of this theorem is based on the following lemma.

\begin{lemma}[{\cite[Lemma 3.5]{PX}}]
	Let sequences $A_{n}, B_n$ be the finite sums
	\[
	A_n=\sum_{k=1}^{n}a_k,~~B_n=\sum_{k=1}^{n}b_k\five(a_n, b_n=o(n^{-p}), \mm{Re}(p)>1)
	\]
	and $A=\lim_{n\rightarrow \infty} A_n, B=\lim_{n\rightarrow\infty}B_n$. Also, we define 
	\[
	H_n(\alpha)=\sum_{k=1}^{n}\frac{1}{k+\alpha}.
	\]
	Then we have
	\[
	\sum_{n=1}^{\infty}\left\{\frac{A_nB}{n+\alpha}-\frac{AB_n}{n+\beta}\right\}=AB(\psi(\beta+1)-\psi(\alpha+1))+A\sum_{n=1}^{\infty} b_nH_{n-1}(\beta)-B\sum_{n=1}^{\infty} a_nH_{n-1}(\alpha),
	\] 
	where $\psi(\alpha+1)$ is the digamma function and $\alpha,\beta\in\Co\setminus\Z_{<0}$.
\end{lemma}

Instead of using this lemma, by considering shuffle normalization of polylogarithm, we have a duality type relation for $\lambda$ function.

\begin{theorem}\label{lamdual}
	For $p, q, r\in\Z_{\geq1}$ and $\Bbbk_r=(k_1,\ldots,k_r)\in\Z_{\geq2}^r$, we have
	\begin{equation}
	\begin{split}\label{formulalamdual}
		\lambda&\left(\{1\}_{q-1},(\overrightarrow{\Bbbk_{r}})_{-};p+1\right)-(-1)^{\mm{wt}(\Bbbk_r)}\lambda\left(\{1\}_{p-1},(\overleftarrow{\Bbbk_{r}})_{-};q+1\right)\\
		&=i^{\mm{dep}(\Bbbk_r)-\mm{wt}(\Bbbk_r)+1}\sum_{j=0}^{r-1}(-1)^{\mm{wt}(\overleftarrow{\Bbbk_j})}\sum_{l=1}^{k_{r-j}-2}(-1)^{l-1}\widetilde{T}\left(\left(\{1\}_{p-1},\overleftarrow{\Bbbk_{j}},l+1\right)^{\dagger}\right)\widetilde{T}\left(\left(\{1\}_{q-1},\overrightarrow{\Bbbk_{r-j-1}},k_{r-j}-l\right)^{\dagger}\right)\\
		&+i^{\mm{dep}(\Bbbk_r)-\mm{wt}(\Bbbk_r)+1}\sum_{j=0}^{r-2}(-1)^{\mm{wt}(\overleftarrow{\Bbbk_{j+1}})}\left\{\widetilde{T}\left(\left(\{1\}_{q-1},\overrightarrow{\Bbbk_{r-j-1}}\right)^{\dagger}\right)\widetilde{T}\left(\left(\left(\{1\}_{p-1},\overleftarrow{\Bbbk_{j+1}}\right)_{-}\textup{\textcircled{$\ast$}}~(1,1)\right)^{\dagger}\right)\right.\\
		&\left.\twen-\widetilde{T}\left(\left(\{1\}_{p-1},\overleftarrow{\Bbbk_{j+1}}\right)^{\dagger}\right)\widetilde{T}\left(\left(\left(\{1\}_{q-1},\overrightarrow{\Bbbk_{r-j-1}}\right)_{-}\textup{\textcircled{$\ast$}}~(1,1)\right)^{\dagger}\right)\right\}.
	\end{split}
    \end{equation}
    Here, $\Bbbk^\dagger$ is the dual index of $\Bbbk$ and 
    \[
    \widetilde{T}\left(\left(\Bbbk_r\right)_{-}\textup{\textcircled{$\ast$}}~(1,1)\right)=\sum_{j=1}^{r-1}\sum_{\substack{k_{j,1}+k_{j,2}=k_j+1\\k_{j,1},k_{j,2}\geq1, k_{1,1}\geq2}}^{}\widetilde{T}(k_1,\ldots, k_{j-1},k_{j,1},k_{j,2},k_{j+1}\ldots,k_{r}).
    \]
\end{theorem}

 \begin{proof}
	We use the variable change similar to \cite{Ya}. By changing variable $\tanh{\left(\frac{t}{2}+\frac{\pi}{4}i\right)}=z$ in \eqref{lambda} and putting $s=p+1$, we have
	\begin{align*}
		\lambda((\{1\}_{q-1},\overrightarrow{\Bbbk_r})_{-};p+1)&=i\int_{i}^{1} \frac{\apoly{\{1\}_{p}}{z}\apoly{(\{1\}_{q-1},\overrightarrow{\Bbbk_r})_{-}}{z}}{z}dz.
	\end{align*}
	For the integrand, by integrating by parts, we have
	\begin{align*}
		i&\frac{\apoly{\{1\}_{p}}{z}\apoly{(\{1\}_{q-1},\overrightarrow{\Bbbk_r})_{-}}{z}}{z}\\
		&\five=i\sum_{j=0}^{r-1}(-1)^{\mm{wt}(\overleftarrow{\Bbbk_j})}\sum_{l=1}^{k_{r-j}-2}(-1)^{l-1}\dv{}{z}\left\{\apoly{\{1\}_{p-1},\overleftarrow{{\Bbbk_{j}}},l+1}{z}\apoly{\{1\}_{q-1},\overrightarrow{{\Bbbk_{r-j-1}}},k_{r-j}-l}{z}\right\}\\
		&+i\sum_{j=0}^{r-2}(-1)^{\mm{wt}(\overleftarrow{\Bbbk_j})}\dv{}{z}\left\{\apoly{\{1\}_{p-1},\overleftarrow{\Bbbk_{j+1}}}{z}\apoly{\{1\}_{q-1},\overrightarrow{\Bbbk_{r-j-1}},1}{z}\right.\\
		&\twen\ten \left.-\apoly{\{1\}_{p-1},\overleftarrow{\Bbbk_{j+1}},1}{z}\apoly{\{1\}_{q-1},\overrightarrow{\Bbbk_{r-j-1}}}{z}\right\}\\
		&+i(-1)^{\mm{wt}(\overleftarrow{\Bbbk_r})}\frac{\apoly{\{1\}_{q}}{z}\apoly{(\{1\}_{p-1},\overleftarrow{\Bbbk_r})_{-}}{z}}{z}.
	\end{align*}
	For the second sum, by normalization of $\mathscr{A}$, we have
	\begin{align*}
		&\apoly{\{1\}_{p-1},\overleftarrow{\Bbbk_{j+1}}}{z}\apoly{\{1\}_{q-1},\overrightarrow{\Bbbk_{r-j-1}},1}{z}-\apoly{\{1\}_{p-1},\overleftarrow{\Bbbk_{j+1}},1}{z}\apoly{\{1\}_{q-1},\overrightarrow{\Bbbk_{r-j-1}}}{z}\\
		&=\apoly{\{1\}_{p-1},\overleftarrow{\Bbbk_{j+1}}}{z}\apoly{\{1\}_{q-1},\overrightarrow{\Bbbk_{r-j-1}}}{z}\apoly{1}{z}\\
		&\five-\apoly{\{1\}_{p-1},\overleftarrow{\Bbbk_{j+1}}}{z}\apoly{\left(\{1\}_{q-1},\overrightarrow{\Bbbk_{r-j-1}}\right)_{-}\textcircled{$\ast$}~(1,1)}{z}\\
		&\five-\apoly{\{1\}_{p-1},\overleftarrow{\Bbbk_{j+1}}}{z}\apoly{\{1\}_{q-1},\overrightarrow{\Bbbk_{r-j-1}}}{z}\apoly{1}{z}\\
		&\five+\apoly{\left(\{1\}_{p-1},\overrightarrow{\Bbbk_{j+1}}\right)_{-}\textcircled{$\ast$}~(1,1)}{z}\apoly{\{1\}_{q-1},\overleftarrow{\Bbbk_{r-j-1}}}{z}\\
		&=\apoly{\left(\{1\}_{p-1},\overleftarrow{\Bbbk_{j+1}}\right)_{-}\textcircled{$\ast$}~(1,1)}{z}\apoly{\{1\}_{q-1},\overrightarrow{\Bbbk_{r-j-1}}}{z}\\
		&\five-\apoly{\{1\}_{p-1},\overleftarrow{\Bbbk_{j+1}}}{z}\apoly{\left(\{1\}_{q-1},{\overrightarrow{\Bbbk_{r-j-1}}}_{-}\right)\textcircled{$\ast$}~(1,1)}{z}.
	\end{align*}
    By \cite[Proposition 3.2]{KT1}, we have
	\[
	\apoly{\Bbbk_r}{1}=i^{\mm{dep}(\Bbbk_r)-\mm{wt}(\Bbbk_r)}\widetilde{T}((\Bbbk_r)^{\dagger}).
	\]
	Hence the theorem follows by integrating with respect to z from $i$ to $1$.
\end{proof}

\begin{remark}
	As the same manner, we obtain
	\begin{align*}
		\psi&\left(\{1\}_{q-1},(\overrightarrow{\Bbbk_{r}})_{-};p+1\right)-(-1)^{\mm{wt}(\Bbbk_r)}\psi\left(\{1\}_{p-1},(\overleftarrow{\Bbbk_{r}})_{-};q+1\right)\\
		&=\sum_{j=0}^{r-1}(-1)^{\mm{wt}(\overleftarrow{\Bbbk_j})}\sum_{l=1}^{k_{r-j}-2}(-1)^{l-1}T(\{1\}_{p-1},\overleftarrow{\Bbbk_{j}},l+1)T(\{1\}_{q-1},\overrightarrow{\Bbbk_{r-j-1}},k_{r-j}-l)\\
		&+\sum_{j=0}^{r-2}(-1)^{\mm{wt}(\overleftarrow{\Bbbk_{j+1}})}\left\{T\left(\{1\}_{q-1},\overrightarrow{\Bbbk_{r-j-1}}\right)T\left(\left(\{1\}_{p-1},\overleftarrow{\Bbbk_{j+1}}\right)_{-}\textup{\textcircled{$\ast$}}~(1,1)\right)\right.\\
		&\left.\twen-T\left(\{1\}_{p-1},\overleftarrow{\Bbbk_{j+1}}\right)T\left(\left(\{1\}_{q-1},\overrightarrow{\Bbbk_{r-j-1}}\right)_{-} \textup{\textcircled{$\ast$}}~(1,1)\right)\right\}.
	\end{align*}
\end{remark}

\begin{example}
	For $\Bbbk_r=(2,2), p=1$ and $q=2$, by Theorem \ref{12lam}, we have
	\begin{align*}
		\mbox{$($LHS of \eqref{formulalamdual} $)$} &= \lambda(1,2,1;2)-\lambda(2,1;3)\\
		&=2i\widetilde{T}(3,3)+3i\widetilde{T}(4,2).
	\end{align*}
	On the other hand, a direct calculation shows that
	\begin{align*}
		\mbox{$($RHS of \eqref{formulalamdual} $)$} &= -i\widetilde{T}((1)~\textup{\textcircled{$\ast$}}~(1,1))^{\dagger})\widetilde{T}(3)+i\widetilde{T}(2)\widetilde{T}(((1,1)~\textup{\textcircled{$\ast$}}~(1,1))^{\dagger})\\
		&=2i\tilde{T}(3,3)+3i\tilde{T}(4,2).
	\end{align*}
	Hence we can see that Theorem \ref{lamdual} holds.
\end{example}

\section*{Acknowledgments}

The author would like to thank my supervisor Professor Hirofumi Tsumura for his kind advice and helpful comments. This work was supported by JST, the establishment of university fellowships towards the creation of science technology innovation, Grant Number JPMJFS2139.


\begin{bibdiv}
		\begin{biblist}
		    \bib{AK}{article}{
		    	author={T. Arakawa},author={M. Kaneko},
		    	title={Multiple zeta values, poly-Bernoulli numbers, and related zeta functions},
		    	journal={Nagoya Math. J.},
		    	volume={153},
		    	date={1999},
		    	number={},
		    	pages={189--209},
		    	issn={},
		    }
	     \bib{AK2}{article}{
	    	author={T. Arakawa},author={M. Kaneko},
	    	title={On multiple $L$-values},
	    	journal={J. Math. Soc. Japan},
	    	volume={56},
	    	date={2004},
	    	number={4},
	    	pages={967--991},
	    	issn={},
	    }
	        \bib{Hof1}{article}{
	        	author={M. E. Hoffman},
	        	title={An odd variant of multiple zeta values},
	        	journal={Commun. Number Theory Phys.},
	        	volume={13},
	        	date={2019},
	        	number={3},
	        	pages={529--567},
	        	issn={},
	        }
		    \bib{KT2}{article}{
		    	author={M. Kaneko},author={H. Tsumura},
		    	title={On multiple zeta values of level two},
		    	journal={Tsukuba J. Math.},
		    	volume={44},
		    	date={2020},
		    	number={},
		    	pages={213-234},
		    	issn={},
		    }
			\bib{KT3}{article}{
				author={M. Kaneko},author={H. Tsumura},
				title={Zeta functions connectiong multiple zeta values and poly-Bernoulli numbers},
				journal={Adv. Stud. Pure Math.},
				volume={84},
				date={2020},
				number={},
				pages={181--204},
				issn={},
			}
					\bib{KT1}{article}{
			author={M. Kaneko},author={H. Tsumura},
			title={Multiple $L$-values of level four, poly-Euler numbers, and related zeta function},
			journal={to appear in Tohoku Math. J. arXiv:2208.05146},
			volume={},
			date={},
			number={},
			pages={},
			issn={},
		}
	\bib{KY}{article}{
	AUTHOR = {M. Kaneko},author={S. Yamamoto},
	TITLE = {A new integral-series identity of multiple zeta values and
		regularizations},
	JOURNAL = {Selecta Math.},
	VOLUME = {24},
	YEAR = {2018},
	NUMBER = {3},
	PAGES = {2499--2521},	
}
			\bib{PX}{article}{
				author={M. Pallewatta},author={C. Xu},
				title={Some results on Arakawa-Kaneko, Kaneko-Tsumura functions and related functions},
				journal={arXiv:1901.07877},
				volume={},
				date={2019},
				number={},
				pages={},
				issn={},
			}
			\bib{Xu}{article}{
				author={C. Xu},
				title={Duality formulas for Arakawa-Kaneko zeta values and related variants},
				journal={Bull. Malays. Math. Sci. Soc.},
				volume={},
				date={2021},
				number={},
				pages={3001--3018},
				issn={},
			}
				\bib{CZ}{article}{
					author={C. Xu},author={J. Zhao},
					title={Variants of multiple zeta values with even and odd summation
						indices},
					journal={Math. Z.},
					volume={300},
					date={2022},
					number={3},
					pages={3109--3142},
					issn={},
			}
		\bib{Ya}{article}{
		AUTHOR = {S. Yamamoto},
		TITLE = {Multiple zeta functions of {K}aneko-{T}sumura type and their
			values at positive integers},
		JOURNAL = {Kyushu J. Math.},
		VOLUME = {76},
		YEAR = {2022},
		NUMBER = {2},
		PAGES = {497--509},
	}
		\end{biblist}
	\end{bibdiv}
\end{document}